\documentclass[12pt]{amsart}
\usepackage{geometry} % see geometry.pdf on how to lay out the page. There's lots.
\usepackage{kpfonts}
\usepackage{amscd}
\usepackage{amsmath}
\usepackage[all,cmtip]{xy}
\usepackage{mathrsfs}
\usepackage{amsthm}
\usepackage[pdfpagelabels]{hyperref}
\usepackage{cite}
\def\beq{\begin{equation}}
\def\eeq{\end{equation}}
\def\beqa{\begin{eqnarray}}
\def\eeqa{\end{eqnarray}}

\newtheorem{thm}{Theorem} [section]
\newtheorem{lem}[thm]{Lemma}

\theoremstyle{definition}

\begin{document}
\title{A quantum space and some associated quantum groups}

\author[Muttalip \"{O}zav\c{s}ar ]{Muttalip \"{O}zav\c{s}ar}
\address{Department
of Mathematics \\ 
Yildiz Technical University \\
Davutpasa-Esenler\\
P.O. Box 34210 \\
Turkey}

\email{mozavsar@yildiz.edu.tr}

\begin{abstract}
In this paper, we first introduce a quantum $n$-space with a cocommutative Hopf algebra structure. Then it is shown that to this quantum $n$-space there corresponds a derivation algebra of $\sigma$-twisted derivations related to some algebra automorphisms $\sigma$ on the quantum $n$-space. Furthermore, we show that this derivation algebra is a noncommutative and non-cocommutative Hopf algebra, namely, a quantum group. Morever, for this the quantum $n$-space, we show how to construct a bicovariant differential calculus related to the derivation algebra.
\end{abstract}
\maketitle

\section{Introduction}
It is well known from \cite{journal-1:1988} that the notion of quantum group is used to specify deformed
versions of some classical Lie groups in the classical differential geometry. In addition, as the most concrete examples of noncommutative spaces, quantum spaces(groups) are also regarded by many as a pattern in generalizing  quantum deformed physics(for example, see [2-4]). In this context many efforts were displayed in order to construct noncommutative differential calculus on quantum spaces(groups) (see [5-9]). 
Constructing comprehensively the scheme of noncommutative differential calculus on a quantum group can be realized by using the setting of Hopf algebra \cite{journal-3:1989}, and this differential calculus is extended to the graded differential Hopf algebra \cite{journal-6:1993}. The main interest in this paper is first to introduce a quantum $n$-space which has noncommutative algebra relations, and as Hopf algebra is cocommutative one. Then, for each generator of this quantum $n$-space, we define a linear mapping which becomes the usual derivation operator(vector field) as the deformation parameter tends to 1. To make this fact algebraically meaningful, an algebra automorphism corresponding to each generator is given in such a way that each linear mapping has a twisted form of the usual Leibniz rule via the corresponding automorphism. So this property enables us to define a noncommutative and non-cocommutative Hopf algebra for the algebra generated by the twisted derivations and the corresponding automorphisms. We  also show that these twisted derivation operators are consistent with the Hopf structure of quantum $n$-space in the sense of their functionality in constructing a bicovariant differential calculus on the quantum $n$-space.  
\section{Preliminaries}
Throughout the paper, the complex numbers $\mathbb{C}$ will be always ground field for all our objects. An associative unital algebra is a triple 
$(\mathcal{A}, m, \eta)$ as a collection of a linear space $\mathcal{A}$, a linear multiplication mapping $m:\mathcal{A}\otimes\mathcal{A}\rightarrow \mathcal{A}$
and a unit mapping $\eta:\mathbb{C} \rightarrow \mathcal{A}$ with the well known axioms $m\circ (m\otimes \mbox{id})=m \circ (\mbox{id}\otimes m)$ and $m \circ (\eta\otimes \mbox{id)}=m\circ (\mbox{id}\otimes \eta)$, where $\mbox{id}$ denotes identity mapping. A coassociative \textit{coalgebra} $({\mathcal{A}}, \Delta, \varepsilon)$ is
defined by inverting all arrows in the diagrams of data and axioms for $(\mathcal{A}, m, \eta)$. This is equivalent to saying that we have a coproduct $\Delta:\mathcal{A}\rightarrow \mathcal{A}\otimes\mathcal{A}$
and a counit $\varepsilon:\mathcal{A}\rightarrow \mathbb{C}$ satisfying two axioms
\begin{equation}\label{1}
(\Delta \otimes \mbox{id})\circ \Delta = (\mbox{id} \otimes \Delta)\circ \Delta
\end{equation}
\begin{equation}\label{2}
(\varepsilon \otimes \mbox{id})\circ\Delta = \mbox{id}=(\mbox{id} \otimes \varepsilon)\circ \Delta.
\end{equation}
 
A \textit{bialgebra} $\mathcal{A}$ is both a unital associative algebra and coalgebra such that $\Delta$ and $\varepsilon$ are both algebra homomorphisms with $\Delta(1_{\mathcal{A}})=1_{\mathcal{A}}\otimes 1_{\mathcal{A}}$
and  $\varepsilon(1_{\mathcal{A}})=1_K $. Finally, a \textit{Hopf algebra} is a bialgebra $\mathcal{A}$ endowed with an algebra homomorphism called ``antipode mapping" $S:\mathcal{A}\rightarrow {\mathcal{A}}^{\text{op}}$ enjoying 
\begin{equation}\label{3}
  m\circ (S\otimes \mbox{id})\circ\Delta = \eta\circ\varepsilon= m\circ(\mbox{id} \otimes S)\circ\Delta,
\end{equation}
where ${\mathcal{A}}^{\text{op}}$ is just the opposite algebra of $\mathcal{A}$.

Let $\mathcal{A}$ be an algebra. A $\mathbb{C}$-vector space $\Omega$ is called a right $\mathcal{A}$-module if there exists a linear mapping $\varphi_R: \Omega\otimes\mathcal{A}\rightarrow \Omega$ with $\varphi_R\circ(\varphi_R\otimes \mbox{id})=\varphi_R\circ(\mbox{id}\otimes m )$ and $\varphi_R\circ(\mbox{id}\otimes\eta)=\mbox{id}$. Similarly,
the vector space $\Omega$ is called a left $\mathcal{A}$-module if there is a linear mapping $\varphi_L:\mathcal{A}\otimes \Omega\rightarrow \Omega$ satisfying the conditions $\varphi_L\circ(\mbox{id}\otimes \varphi_L)=\varphi_L\circ(m\otimes \mbox{id})$ and $\varphi_L\circ(\eta\otimes\mbox{id})=\mbox{id}$. If the $\Omega$ is both a right $\mathcal{A}$-module with the right action $\varphi_R$ and a left $\mathcal{A}$-module with the left action $\varphi_L$ such that they commute, that is, $\varphi_L \circ (\mbox{id}\otimes \varphi_R)=\varphi_R \circ (\varphi_L\otimes \mbox{id})$, then we call $\Omega$ an $\mathcal{A}$-bimodule.
Let $\mathcal{A}$ be a Hopf algebra. A \textit{right comodule} over $\mathcal{A}$ is a $\mathbb{C}$-vector space $\Omega$ equipped with a linear mapping
$\Delta_R : \Omega \longrightarrow \Omega \otimes \mathcal{A}$ satisfying 
\begin{equation}
\begin{aligned}
&(\Delta_R \otimes \mbox{id}) \circ \Delta_R  = (\mbox{id} \otimes \Delta_{\mathcal{A}}) \circ \Delta_R, \\
&\textsl{m} \circ(\mbox{id} \otimes \varepsilon_{\mathcal{A}}) \circ \Delta_R = \mbox{id}.
\end{aligned} 
\end{equation}
A \textit{left comodule} over $\mathcal{A}$ is a $\mathbb{C}$-vector space $\Omega$ equipped with a linear mapping
$\Delta_L : \Omega \longrightarrow \mathcal{A} \otimes \Omega $ such that 
\begin{equation}
\begin{aligned}
&(\mbox{id} \otimes \Delta_L) \circ \Delta_L=(\Delta_{\mathcal{A}} \otimes \mbox{id}) \circ \Delta_L, \\
 & \textsl{m} \circ(\varepsilon_{\mathcal{A}} \otimes \mbox{id})\circ \Delta_L = \mbox{id}. 
\end{aligned} 
\end{equation}
Let $\Omega$ be left and right comodule over $\mathcal{A}$ with the relevant linear mappings $\Delta_L$ and $\Delta_R$. If $\Delta_L$ and $\Delta_R$ hold $( \mbox{id} \otimes \Delta_R )\circ \Delta_L=(\Delta_L \otimes \mbox{id})\circ \Delta_R$, that is, they commute, then one says that $\Omega$ is a \textit{bicomodule} over $\mathcal{A}$. A \textit{bicovariant bimodule} $\Omega$ is a bicomodule $\Omega$ such that the relevant mappings $\Delta_L$ and $\Delta_R$  hold the condition of compatibility with module structure of $\Omega$:
\begin{equation*}
\begin{aligned}
&\Delta_L(apb) =\Delta_{\mathcal{A}}(a)\Delta_L(p)\Delta_{\mathcal{A}}(b), \\
&\Delta_R(apb) =\Delta_{\mathcal{A}}(a)\Delta_R(p)\Delta_{\mathcal{A}}(b)
\end{aligned}
\end{equation*}
for all $ a,b \in \mathcal{A} $ and $ p \in \Omega $.

Let $\Omega^1$ be a bimodule over an algebra $\mathcal{A}$ and ${\sf d}$ be a linear mapping from $\mathcal{A}$ to $\Omega^1$. The pair $(\Omega^1, {\sf d})$ is called a \textit{first order differential calculus} over $\mathcal{A}$ if the Leibniz rule ${\sf d}(ab)={\sf d}(a)b+a{\sf d}(b)$ holds for all $a,b \in \mathcal{A}$. Note that $\Omega^1$ is the linear span of elements $a{\sf d}(b)$ or ${\sf d}(a')b', a,a', b,b' \in \mathcal{A}$.
Morever, we say that a first order differential calculus $(\Omega^1, {\sf d})$ is bicovariant if the bimodule $\Omega^1$ is bicovariant together with the coactions $\Delta_L({\sf d}(a))=(\mbox{id}\otimes {\sf d})(\Delta_{\mathcal{A}}(a))$ and $\Delta_R({\sf d}(a))=({\sf d}\otimes \mbox{id}) (\Delta_{\mathcal{A}}(a)), a \in \mathcal{A}$, where $\Delta_L|_{\mathcal{A}}=\Delta_{\mathcal{A}}$ and
 $\Delta_R|_{\mathcal{A}}=\Delta_{\mathcal{A}}$(see \cite{journal-3:1989}). Let $\Omega^n$ be the space of differential $n$-forms. Let any element of $\Omega^n$ be denoted by $a{\sf d}(a_1)\wedge{\sf d}(a_2)\wedge...\wedge{\sf d}(a_n)$ or ${\sf d}(a_1)\wedge{\sf d}(a_2)\wedge...\wedge{\sf d}(a_n)a$ for $a$,$a_i$'s$\in \mathcal{A}$ where the multiplication $\wedge$ is defined as $u\wedge v \in \Omega^{m+n}$ for $u\in \Omega^{m}$ and $v\in \Omega^{n}$. Thus, the exterior algebra of all higher order differential forms (or differential graded algebra) is a $\mathbb{N}_0$-graded algebra $\Omega^{\wedge}=\bigoplus_{n\geq 0}\Omega^n$, $\Omega^0={\mathcal{A}}$, with the exterior differential mapping ${\sf d}:\Omega^{\wedge}\rightarrow \Omega^{\wedge}$ of degree 1 satisfying ${\sf d}^2=0$ and the graded Leibniz rule ${\sf d}(w_1\wedge w_2)={\sf d}(w_1)\wedge w_2+(-1)^n w_1 \wedge{\sf d}(w_2), \ w_1 \in \Omega^n, \ w_2 \in \Omega^{\wedge}$. Thus the definition of bicovariant first order differential calculus can be extented to the differential graded algebra $\Omega^{\wedge}$ in a similar manner. Morever, a bicovariant differential algebra $\Omega^{\wedge}$ has a Hopf algebra structure induced by the coproduct $\hat{\Delta}=\Delta_L+\Delta_R$ (for more details see \cite {journal-6:1993, journal-7:1999}).

Let $\sigma$ be an automorphism of an algebra $\mathcal{A}$. For ${\mathcal{A}}$, a $\sigma$-derivation is a linear mapping $\partial:{\mathcal{A}}\rightarrow {\mathcal{A}}$ with $ \partial(ab) = \partial(a)b + \sigma(a)\partial(b)$ for all $a, b \in \mathcal{A}$.

Let us recall that in Manin's terminology a quantum $n$-space is a finitely generated quadratic algebra with noncommutative algebra relations \cite{journal-8:1988, journal-9:1989}. That is, a quantum $n$-space is an associative $\mathbb{C}$-algebra generated by $x_1,x_2,...,x_n$ satisfying the following commutation relations:
\begin{equation}\label{7}
x_ix_j=q_{ij}x_jx_i,~i,j=1,2,...,n,
\end{equation}
where $q_{ij}$'s are nonzero complex numbers with $q_{ij}^{-1}=q_{ji}$.

\section{A Quantum $n$-Space with Hopf Algebra Structure}
This section is concerned with construction of a quantum $n$-space on which a Hopf algebra can be given. The construction of commutation relation of this quantum $n$-space is given by using some means such as bicharacter and 2-cocycle on additive group $\mathbb{Z}^n$ as used in \cite{journal-10:2000}. However, our quantum $n$-space differs from that of \cite{journal-10:2000} in two aspects. First, we make a bicharacter which yields a commutation relation that is not only based on the powers of generators but also their indices, and second, we take a grouplike generator into account.
Let $\alpha=(\alpha_1,...,\alpha_n), \beta=(\beta_1,...,\beta_n)\in\mathbb{Z}^n$ be any two integer $n$-tuples, and consider a mapping $\ast:\mathbb{Z}^n \times \mathbb{Z}^n \rightarrow \mathbb{Z}$ defined through
\begin{equation}\label{8}
\alpha \ast \beta=\sum^{n-1}_{j=1}\sum_{i>j}(j-i)\alpha_i\beta_j.
\end{equation}
One can easily show that the mapping $\ast$ holds the following distributive laws:
\begin{eqnarray}
(\alpha+\beta)\ast \gamma&=&\alpha\ast \gamma+\beta\ast\gamma \nonumber, \\
\alpha\ast(\beta+\gamma)&=&\alpha\ast\beta+\alpha\ast\gamma \nonumber.
\end{eqnarray}
The mapping $\ast$ also satisfies
\begin{equation*}
\varepsilon_i\ast\beta =\sum_{s<i}(s-i)\beta_s~(1\leq i \leq n),
\end{equation*}
\begin{equation*}
\beta\ast\varepsilon_i =\sum_{s>i}(i-s)\beta_s~(1\leq i \leq n),
\end{equation*}
\begin{equation*}
(\varepsilon_i-\varepsilon_{i+1})\ast\beta=\sum^{i}_{s=1}\beta_s~(1\leq i < n),
\end{equation*}
\begin{equation*}
\beta\ast(\varepsilon_i-\varepsilon_{i+1})=-\sum^{n}_{s=i+1}\beta_s~(1\leq i < n),
\end{equation*}
where $\varepsilon_i=(\delta_{1i},...,\delta_{ni})~(1\leq i\leq n)$ is a basis of $\mathbb{Z}^n$ as a $\mathbb{Z}$-module.
Let $q$ be a nonzero complex constant. Now we define a mapping 
$\eta:\mathbb{Z}^n\times \mathbb{Z}^n\rightarrow \mathbb{C}^*$ as follows
\begin{equation}\label{9}
\eta(\alpha,\beta)=q^{\alpha\ast\beta-\beta\ast \alpha}.
\end{equation}
In particular, we can verify that $\eta(\varepsilon_i,\varepsilon_j)=q^{j-i}$. In fact, since the following properties hold
\begin{equation}\label{10}
\eta(\alpha+\beta,\gamma)=\eta(\alpha, \gamma)\eta(\beta,\gamma)
\end{equation}
\begin{equation}\label{11}
\eta(\alpha,\beta+\gamma)=\eta(\alpha, \beta)\eta(\alpha,\gamma)
\end{equation}
\begin{equation}\label{12}
\eta(\alpha,0)=1=\eta(0,\alpha)
\end{equation}
\begin{equation}\label{13}
\eta(\alpha,\beta)\eta(\beta,\alpha)=1=\eta(\alpha,\alpha),
\end{equation}
the mapping $\eta$ is a bicharacter of the additive group $\mathbb{Z}^n$. We are now in a position to give our quantum $n$-space. This quantum $n$-space is generated by $x_1,...,x_n$ having commutation relations which we write as
\begin{equation}\label{14}
x_ix_j = \eta(\varepsilon_i,\varepsilon_j)x_jx_i,~ 1\leq i,j\leq n.
\end{equation}
This space or, rather, the polynomial function ring is formally defined by the following ring
\begin{equation}\label{15}
{\mathcal{A}}_q(n)=\mathbb{C}\left\langle x_1,...,x_n\right\rangle/\mathcal{I},
\end{equation}
where $\mathbb{C}\left\langle x_1,...,x_n\right\rangle$ means an associative $\mathbb{C}$-algebra freely generated by $x_1,...,x_n$ and $\mathcal{I}$ is an ideal generated by the relations (\ref{14}).
This is a deformation of the usual commutative space corresponding to the case $q=1$. Let $\mathbb{Z}^n_+$ show  
the set of nonnegative-integer $n$-tuples and $x^{\alpha}=x_1^{\alpha_1}\cdot\cdot\cdot x_n^{\alpha_n}$ any nonzero monomial in ${\mathcal{A}}_q(n)$ where $\alpha=(\alpha_1,...,\alpha_n) \in \mathbb{Z}^n_+ $. So we assume that the quantum $n$-space has a PBW basis whose elements are of the ordered form $x^{\alpha}=x_1^{\alpha_1}\cdot\cdot\cdot x_n^{\alpha_n}$. 
That is, the vector space ${\mathcal{A}}_q(n)$ can be given as a direct sum of the vector subspaces consisting of the above ordered monomials of homogeneity degree $m$. From now on, our all linear mappings will be defined by taking into account the above ordering. Morever, we easily derive the commutation relation between two monomials $x^{\alpha},x^{\beta}$ with the help of $\eta$ as follows: 
\begin{equation*}
x^{\alpha}x^{\beta}= \eta(\alpha,\beta)x^{\beta}x^{\alpha}.
\end{equation*}
We note at this point that since $\eta$ holds the conditions $(\ref{10}-\ref{13})$, $\eta$-commutativity is well defined, and $\eta$ is a 2-cocycle on the additive group $\mathbb{Z}^n$, meaning that it satisfies
\begin{equation}\label{cocyc}
\eta(\alpha,\beta)\eta(\alpha+\beta,\gamma)=\eta(\beta,\gamma)\eta(\alpha, \beta+\gamma),
\end{equation}
which ensures the compatibility of $\eta$-commutativity with associativity rule. So the above construction shows that the quantum $n$-space ${\mathcal{A}}_q(n)$ is an associative $\eta$-commutative  algebra if we set that $x^0=1$ and $x^{\alpha}=0$ for $\alpha\notin \mathbb{Z}^n_{+}$.
For more details on more general algebras such as $\Gamma-$graded $\eta-$commutative algebras, we refer the reader to the series of papers \cite{journal-11:1979,journal-12:1983}. 
The multiplication of homogeneous elements in the tensor products of such $\Gamma-$graded $\eta-$commutative algebras can be defined in the manner $(a\otimes b)(c\otimes d)=\eta(g_b,g_c)ac\otimes bd$ where $g_b$ and $g_c$ are degree elements in $\Gamma$ assigned to the degrees for elements $b$ and $c$, respectively. But, the usual multiplication in ${\mathcal{A}}_q(n)\otimes {\mathcal{A}}_q(n)$ must be used in the setting of Hopf algebra given below, while the algebra over it is $\Gamma-$ graded $\eta-$commutative. In other words, if we use this multiplication instead of the usual one in the setting of Hopf algebra, then the coproduct $\Delta$ does not leave invariant the commutation relation (\ref{14}). This property makes the quantum space defined above more interesting to us.
From now on, for multiplication in ${\mathcal{A}}_q(n)\otimes {\mathcal{A}}_q(n)$, we use  the usual one defined by $(a\otimes b)(c\otimes d)=ac\otimes bd$ throughout the paper.
\begin{thm}
Let us extend ${\mathcal{A}}_q(n)$ by the inverse $x_1^{-1}$. Given two algebra homomorphisms $\Delta:{\mathcal{A}}_q(n)\rightarrow {\mathcal{A}}_q(n)\otimes {\mathcal{A}}_q(n)$, $\varepsilon:{\mathcal{A}}_q(n)\rightarrow \mathbb{C}$ and an algebra antihomomorphism $S:{\mathcal{A}}_q(n)\rightarrow {\mathcal{A}}_q(n)$ , acting on the generators $x_1,x_2,...,x_n$ as follows: 
\begin{equation}\label{16}
\begin{aligned}
&\Delta(x_1^{\pm1})=x_1^{\pm1}\otimes x_1^{\pm1},\\
&\Delta(x_i)=x_i\otimes x_1+x_1\otimes x_i,~1<i\leq n, \\
&\varepsilon(x_1)= 1,~ \varepsilon(x_i)=0,~1<i\leq n, \\
&S(x_1)=x_1^{-1},~S(x_i)=-x_1^{-1}x_ix_1^{-1},~1<i\leq n,
\end{aligned}
\end{equation}
then the extended quantum space ${\mathcal{A}}_q(n)$ is a cocommutative Hopf algebra under the mappings $\Delta$, counit $\varepsilon$ and antipode $S$.
\end{thm}
\begin{proof}
For any basis element $x^{\alpha}$ of the extended quantum space ${\mathcal{A}}_q(n)$, let $\alpha\in \mathbb{Z}\times \mathbb{Z}^{n-1}_{+}$. Thus, it is clear that the extended ${\mathcal{A}}_q(n)$ is again $\eta$-commutative algebra in the assumption that $x^0=1$ and $x^{\alpha}=0$ if the last ($n$-1)-tuple of $\alpha$ $(\alpha_{n-1},...,\alpha_n)\notin \mathbb{Z}^n_{+}$. So, it must be first checked whether the mappings $\Delta, \varepsilon$ and $S$ leave invariant the commutation relation (\ref{14}). For $\varepsilon$, it is clear. For $\Delta$, we shall check whether $\Delta(x_ix_j-\eta(\varepsilon_i,\varepsilon_j)x_jx_i)=0$ holds for all $i,j=1,2,...,n$. It is seen in the following:
\begin{equation}\label{17}
\begin{aligned}
\Delta(x_ix_j)&=(x_i\otimes x_1+x_1\otimes x_i)(x_j\otimes x_1+x_1\otimes x_j)\\
&=(x_ix_j\otimes x_1^2+x_ix_1\otimes x_1x_j+x_1x_j\otimes x_ix_1+x_1^2\otimes x_ix_j\\
&=\eta(\varepsilon_i,\varepsilon_j)(x_jx_i\otimes x_1^2+x_1x_i\otimes x_jx_1+x_jx_1\otimes x_1x_i+x_1^2\otimes x_jx_i)\\
&=\eta(\varepsilon_i,\varepsilon_j)\Delta(x_j)\Delta(x_i)
\end{aligned}
\end{equation}
for $1<i,j\leq n$ with $i\neq j$ and
\begin{eqnarray}
\Delta(x_1x_i)&=&(x_1\otimes x_1)(x_i\otimes x_1+x_1\otimes x_i)\nonumber \\
&=&x_1x_i\otimes x_1^2+x_1^2\otimes x_1x_i \nonumber \\
&=&\eta(\varepsilon_1,\varepsilon_i)(x_ix_1\otimes x_1^2+x_1^2\otimes x_ix_1)\nonumber \\
&=&\eta(\varepsilon_1,\varepsilon_i)\Delta(x_i)\Delta(x_1) \nonumber
\end{eqnarray}
and
\begin{eqnarray}
\Delta(x_1^{-1}x_i)&=&(x_1^{-1}\otimes x_1^{-1})(x_i\otimes x_1+x_1\otimes x_i)\nonumber \\
&=&x_1^{-1}x_i\otimes 1+1\otimes x_1^{-1}x_i \nonumber \\
&=&\eta(\varepsilon_i,\varepsilon_1)(x_ix_1^{-1}\otimes 1+1\otimes x_ix_1^{-1})\nonumber \\
&=&\eta(\varepsilon_i,\varepsilon_1)\Delta(x_i)\Delta(x_1^{-1}) \nonumber
\end{eqnarray}
for $1<i\leq n$.
From the definition of $\Delta$, it is readily apparent that $\Delta$ holds the rule of cocommutativity $\Delta=\tau \circ \Delta$, where $\tau$ is the twisting mapping defined by $\tau(a\otimes b)=b\otimes a$. From the actions of $\Delta, \varepsilon, S$ on the generators, it is also clear that the mappings $\Delta, \varepsilon,S$ fulfill the Hopf algebra axioms (\ref{1}-\ref{3}).
Thus we can note that $S$ is an anti-homomorphism at the level of coalgebra, meaning that
\begin{equation}\label{18}
\tau \circ (S\otimes S)\circ \Delta=\Delta \circ S.
\end{equation}
 Indeed, from (\ref{16}),  
\begin{equation}\label{19} 
\begin{aligned}
\left (\tau \circ (S\otimes S)\circ \Delta\right)(x_i) &=\left(\tau \circ (S\otimes S)\right )(x_i\otimes x_1+x_1\otimes x_i), \\
&= \tau\left (-x_1^{-1}x_ix_1^{-1}\otimes x_1^{-1}-x_1^{-1} \otimes x_1^{-1}x_ix_1^{-1}\right )\\
&= -x_1^{-1} \otimes x_1^{-1}x_ix_1^{-1}-x_1^{-1}x_ix_1^{-1}\otimes x_1^{-1} 
\end{aligned}
\end{equation}
Also
\begin{equation}\label{20}
\begin{aligned}
\left(\Delta \circ S\right)(x_i)&=\Delta \left(-x_1^{-1}x_ix_1^{-1}\right) \\
&=-\left( x_1^{-1}\otimes x_1^{-1}  \right)\left( x_i\otimes x_1+x_1\otimes x_i \right)\left( x_1^{-1}\otimes x_1^{-1} \right) \\
&= -x_1^{-1} \otimes x_1^{-1}x_ix_1^{-1}-x_1^{-1}x_ix_1^{-1}\otimes x_1^{-1}.
\end{aligned}
\end{equation}
 We see from (\ref{19}) and (\ref{20}) that the equality (\ref{18}) exists for the generator $x_i,~ 1<i\leq n$. It is also shown in a similar manner that $ \left(\tau \circ (S\otimes S)\circ \Delta\right)(x_1)=\left(\Delta \circ S\right)(x_1)$.
Finally we observe equalities $ \varepsilon \circ S=\varepsilon $ and $S^2=\mbox{id}$.
\end{proof}

\section{Derivation Algebra on ${\mathcal{A}}_q(n)$}
In this section, we will show that, in presence of automorphisms, some derivation operators on ${\mathcal{A}}_q$ with deformed Leibniz rules given by the automorphism are related with a bicovariant differential calculus on ${\mathcal{A}}_q$.
Let us define a linear mapping $\partial_q/ \partial x_i$ of ${\mathcal{A}}_q$, defined through 
\begin{equation}\label{21}
\frac{\partial_q}{\partial x_i}\left(x^\alpha\right)=\eta(\stackrel{\leftarrow}{\alpha_i},\varepsilon_i )\alpha_ix^{\alpha-\varepsilon_i} \ (i=1,...,n),
\end{equation}
where $\stackrel{\leftarrow}{\alpha_i}=(\alpha_1,..\alpha_{i-1},0,...,0)$ with $\stackrel {\leftarrow}{\alpha_1}=0$. Note that $\frac{\partial_q}{\partial x_i}\left(1\right)=0 $, and even if we extend ${\mathcal{A}}_q$ by $x_1^{-1}$, the the algebra is again $\eta$-commutative. This says that the rule \eqref{21} can be also defined for $x_1^{z}$ where $z$ is a negative integer, for example, $\frac{\partial_q}{\partial x_1}(x_1^{z})=zx_1^{z-1}$ as the usual one.
For simplicity of notation, we let $\partial_i$ denote briefly $\partial_q/ \partial x_i$. From the definition (\ref{21}) one has the following commutation rule between two mappings $\partial_i , \partial_j$ with respect to their composition as follows:
\begin{equation}\label{22}
\partial_i \partial_j=\eta(\varepsilon_i,\varepsilon_j)\partial_j \partial_i.
\end{equation} 
We also note that the mapping $\partial_i$ reduces to the usual partial derivative operator with respect to $x_i$ when $q\rightarrow 1$. It is intriguing at this point to ask whether the mappings $\partial_i, i=1,...,n$ are deformed derivations acting on ${\mathcal{A}}_q(n)$. To answer this, we investigate what kind of Leibniz rule exists for the each mapping. Really, each $\partial_{i}$ has the following deformed Leibniz rule when it acts on $x^{\alpha}g$, where $g$ is any element of ${\mathcal{A}}_q$:
\begin{equation}\label{der}
\partial_i(x^\alpha g)=\partial_i(x^\alpha)g+\sigma_i(x^{\alpha})\partial_i(g),
\end{equation}
where for any $\beta \in \mathbb{Z}^{n}$, the mapping $\sigma_{\beta}:{\mathcal{A}}_q(n)\rightarrow {\mathcal{A}}_q(n)$ is an algebra automorphism acting on $x^{\alpha}$
\begin{equation}\label{ato}
\sigma_{\beta}(x^{\alpha})=\eta(\alpha, \beta)x^{\alpha} \ \ \text{for all}\ \ x^{\alpha}\in {\mathcal{A}}_q(n),
\end{equation}
and $\sigma_i$ denotes $\sigma(\varepsilon_i),i=1,...,n$. Note that this rule can be also defined for negative powers of $x_1$.

From the properties of $\eta$, it is obviously seen that for all $\alpha, \beta\in \mathbb{Z}^n$, the relation between automorphisms $\sigma_{\alpha},\sigma_{\beta}$ appears as commutative one
\begin{equation}\label{24}
\sigma_{\alpha}\sigma_{\beta}=\sigma_{\alpha +\beta}=\sigma_{\beta}\sigma_{\alpha}  
\end{equation}
\begin{equation}\label{25}
\sigma(0)=1, \  \ \sigma_{\alpha}^{-1}=\sigma_{-\alpha},
\end{equation}
and the relation of an automorphism $\sigma(\alpha)$ with the mapping $\partial_i, i=1,...,n$ is of the form
\begin{equation}\label{26}
\sigma_{\alpha}\partial_i= \eta(\alpha, \varepsilon_i)\partial_i\sigma_{\alpha} 
\end{equation}

Considering the above constructions, we give an algebra ${\mathcal{D}}_q(2n)$ freely generated by $\partial_1,....,\partial_n, \sigma_1,...,\sigma_n$ enjoying the relations \eqref{22}, \eqref{24} and \eqref{26}. Here we note that ${\mathcal{D}}_q(2n)$ is a deformation of the algebra ${\mathcal{D}}(n)$ generated by the usual partial derivations $\partial_1,...,\partial_n$ because the mappings $\sigma_i$ reduce to the idendity mapping when $q\rightarrow 1$. It is well known that $\mathcal{D}(n)$ has a Hopf algebra structure via the coproduct $\Delta$, counit $\varepsilon$ and antipode $S$ acting on the generators as $\Delta(\partial_i)=\partial_i\otimes \mbox{id}+\mbox{id}\otimes \partial_i$, $\varepsilon(\partial_i)=0, S(\partial_i)=-\partial_i$. Based on the above investigations, one can investigate whether the deformed derivation algebra $\mathcal{D}_q(2n)$ has a Hopf algebra structure as a deformation of the Hopf algebra structure for ${\mathcal{D}}(n)$. In fact, the answer is hidden in the structure of automorphisms and Leibniz rule of the derivations. That is, using the definitions $m(\Delta(\partial_a)(f\otimes g)):= \partial_a(f g)$ and $m(\Delta(\sigma_a)(f\otimes g)):=\sigma_a(fg)= \sigma_a(f)\sigma_a(g)$ together with (\ref{der}) implies that the coproduct $\Delta$ acts on the generators as $\Delta(\sigma_i)=\sigma_i\otimes \sigma_i$ and $\Delta(\partial_i)=\partial_i\otimes id+\sigma_i\otimes \partial_i$.
 Therefore we can give the following theorem:
\begin{thm}
The derivation algebra $\mathcal{D}_q(2n)$ can be equipped with a Hopf algebra structure due to the coproduct $\Delta:{\mathcal{D}}_{q}(2n)\rightarrow {\mathcal{D}}_{q}(2n)\otimes \mathcal{D}_q(2n)$:
\begin{equation}\label{28}
\begin{aligned}
&\Delta(\sigma_i)=\sigma_i\otimes \sigma_i\\
&\Delta(\partial_i)=\partial_i\otimes \mbox{id}+\sigma_i\otimes \partial_i,
\end{aligned}
\end{equation}
the counit $\varepsilon:{\mathcal{D}}_{q}(2n)\rightarrow \mathbb{C}$
\begin{equation}\label{29}
\varepsilon(\sigma_i)=1, \ \varepsilon(\partial_i)=0,
\end{equation}
and the antipode $S:{\mathcal{D}}{q}(2n)\rightarrow {\mathcal{D}}{q}(2n)$
\begin{equation}\label{30}
S(\sigma_i)=\sigma_i^{-1}, \ S(\partial_i)=-\sigma_i^{-1}\partial_i,
\end{equation}
where $i=1,...,n$.
\end{thm}
\begin{proof}
 It is first shown that that the commutation relations (\ref{22}), (\ref{24}) and (\ref{26}) are preserved under the actions of $\Delta, \varepsilon$, and $S$. This is clearly seen for $\varepsilon$ and $S$. To check it for $\Delta$, we apply $\Delta$ to the commutation relation $\partial_i\sigma_i-\eta(\varepsilon_i,\varepsilon_j)\sigma_j\partial_i=0$, and we show $\Delta(\partial_i\sigma_i-\eta(\varepsilon_i,\varepsilon_j)\sigma_j\partial_i)=0$ holds as follows:
\begin{equation*}
\begin{aligned}
\Delta(\partial_i\sigma_i-\eta(\varepsilon_i,\varepsilon_j)\sigma_j\partial_i)&= \Delta(\partial_i)\Delta(\sigma_j)-\eta(\varepsilon_i,\varepsilon_j)\Delta(\sigma_j)\Delta(\partial_i) \\
&=(\partial_i\otimes \mbox{id}+\sigma_i\otimes \partial_i)(\sigma_j\otimes \sigma_j) \\
&-\eta(\varepsilon_i,\varepsilon_j)(\sigma_j\otimes \sigma_j)(\partial_i\otimes \mbox{id}+\sigma_i\otimes \partial_i) \\
&=(\partial_i\sigma_j-\eta(\varepsilon_i,\varepsilon_j)\sigma_j\partial_i)\otimes \sigma_j+ \sigma_i\sigma_j\otimes(\partial_i\sigma_j-\eta(\varepsilon_i,\varepsilon_j)\sigma_j\partial_i) \\
&=0.
\end{aligned}
\end{equation*}
Similarly, one can observe that the other relations are also invariant under the action of $\Delta$. Finally, from (\ref{28}), (\ref{29}) and (\ref{30})  it is clear that the above mappings enjoy the Hopf algebra axioms.
\end{proof}

\subsection{A bicovariant differential calculus over $\mathcal{A}_q(n)$}
This subsection is devoted to the construction of a bicovariant differential calculus of $\mathcal{A}_q(n)$ related with $\sigma_i$-derivations $\partial_i,i=1,...,n$.
\begin{thm} 
Let $\Omega^1$ be an ${\mathcal{A}}_{q}(n)$-bimodule with basis elements ${\sf d}x_1,...,{\sf d}x_n $ whose relations with the generators $x_1,...,x_n$ are of the form 
\begin{equation}\label{31}
x_i{\sf d}x_j=\eta(\varepsilon_i,\varepsilon_j){\sf d}x_ix_j, \ i,j=1,...,n. 
\end{equation}
If a mapping ${\sf d}:\mathcal{A}_q(n)\rightarrow \Omega^1$ is defined by
\begin{equation}\label{32}
{\sf d}(f)={\sf d}x_1\partial_1(f)+\cdot \cdot \cdot +{\sf d}x_n\partial_n(f)=
\end{equation}
then the pair $(\Omega^1,{\sf d})$ is a first order differential calculus on $\mathcal{A}_q(n)$.
\end{thm}

\begin{proof}
By ${\sf d}:{\mathcal{A}}_{q}(n)\rightarrow \Omega^1$, it is clear that ${\sf d}(x_i)$ =${\sf d}x_i$. Thus, to show that the pair $(\Omega^1, {\sf d})$ is a first order differential calculus, it is sufficient to prove that ${\sf d}(fg)={\sf d}(f)g+f{\sf d}(g)$, where $f=x^{\alpha}$ and $g$ is any element of ${\mathcal{A}}_{q}(n)$. For this goal we first need the following commutation relation obtained by \eqref{31}:
\begin{equation}\label{33}
{\sf d}x_i\sigma_i(f)=f{\sf d}x_i,
 \end{equation}
 By substituting $fg$ into \eqref{32} and taking into account \eqref{21}, we have the following
\begin{equation}\label{34}
{\sf d}(fg)={\sf d}x_1( \partial_1(f)g+\sigma_1(f)\partial_1(g))+\cdot \cdot \cdot +{\sf d}x_n( \partial_n(f)g+\sigma_n(f)\partial_n(g)).
\end{equation}
Making use \eqref{33} in \eqref{34}, we have 
\begin{equation}\label{35}
{\sf d}(fg)=({\sf d}x_1\partial_1+\cdot \cdot \cdot+{\sf d}x_n\partial_n)(f)g+f({\sf d}x_1\partial_1+\cdot \cdot \cdot+{\sf d}x_n\partial_n)(g),
\end{equation}
which turns out that ${\sf d}(fg)= {\sf d}(f)g+f{\sf d}(g)$. Thus the pair $(\Omega^1, {\sf d})$ is a first order differential calculus on ${\mathcal{A}}_{q}(n)$. 
Morever, since ${\mathcal{A}}_{q}(n)$ is a Hopf algebra, using Woronowicz's approach, we can define the right covariant bimodule structure by a mapping $\Delta_R:\Omega^1 \rightarrow \Omega^1 \otimes {\mathcal{A}}_{q}(n)$, defined as $\Delta_R=({\sf d}\otimes \mbox{id})\circ \Delta $, and the left covariant bimodule structure by a mapping $\Delta_L:\Omega^1 \rightarrow \mathcal{A}_q(n) \otimes\Omega^1$, given by $\Delta_L=(\mbox{id}\otimes {\sf d})\circ \Delta$. Note that $\Delta_R$ and $\Delta_L$ act on ${\mathcal{A}}_{q}(n)$ as the coproduct given in \eqref{16}. Finally, it is remain to show that the mappings $\Delta_L$ and $\Delta_R$ preserve the commutation relations \ref{31}. For example, in the case $1<i \leq n, 1\leq j \leq n$, for $\Delta_R$, we can see this from the following straightforward calculation:
\begin{equation}
\begin{aligned}
\Delta_R(x_i{\sf d}x_j)&=(x_1\otimes x_i+x_i\otimes x_1)({\sf d}x_1\otimes x_j+{\sf d}x_j\otimes x_1) \\
&=x_1{\sf d}x_1\otimes x_ix_j+x_1{\sf d}x_j\otimes x_ix_1+x_i{\sf d}x_1\otimes x_1x_j+x_i{\sf d}x_j\otimes x_1x_1 \\
&=\eta(\varepsilon_i,\varepsilon_j)(({\sf d}x_1\otimes x_j+{\sf d}x_j\otimes x_1)(x_1\otimes x_i+x_i\otimes x_1)\\
&=\eta(\varepsilon_i,\varepsilon_j)\Delta_R({\sf d}x_j)\Delta_R(x_i). 
\end{aligned}
\end{equation}
In the other cases, it can be similarly shown for both $\Delta_R$ and $\Delta_L$. Thus the differential calculus $(\Omega^1, {\sf d})$ is a bicovariant one on ${\mathcal{A}}_{q}(n)$.
\end{proof}
Note that the action of ${\sf d}$ on a negative integer power of $x_1$ is computed by using ${\sf d}(x_1^{-1}):=-{\sf d}(x_1)x_1^{-2}$. 
Based on the above differential calculus $(\Omega^1, {\sf d})$, one can extend the differential operator ${\sf d}$ to the exterior differential operator with the well known properties
\begin{equation}\label{36}
\mathcal{A}\cong \Omega_{0}\stackrel{{\sf d}}{\rightarrow}\Omega_{1}\stackrel{{\sf d}}{\rightarrow}\cdot \cdot \cdot\stackrel{{\sf d}}{\rightarrow}\Omega_{n}\stackrel{{\sf d}}{\rightarrow}\Omega_{n+1}\stackrel{{\sf d}}{\rightarrow}\cdot \cdot \cdot
\end{equation}
\begin{equation}\label{37}
{\sf d}\circ {\sf d}:={\sf d}^2 = 0
\end{equation} 
\begin{equation}\label{38}
 {\sf d}(u\wedge v) = ({\sf d}u)\wedge v + (-1)^{k}u \wedge({\sf d} v), 
\end{equation}
where $u \in \Omega_k $ and $\Omega_{k}$ is the space of differential $k$-forms. Thus, taking into account \eqref{31} with \eqref{37} and\eqref{38}, we get the relation of differentials ${\sf d}x_i$ and ${\sf d}x_j, i,j=1,...,n$ as follows:
\begin{equation}\label{39}
{\sf d}x_i\wedge {\sf d}x_j=(\delta_{ij}-\eta(\varepsilon_i,\varepsilon_j)){\sf d}x_j\wedge {\sf d}x_i.
\end{equation}
At this position we note that the above relation is consistent with the nilpotency rule \eqref{37} when we take into account ${\sf d}:=({\sf d}x_1\partial_1+\cdot \cdot \cdot +{\sf d}x_n\partial_n)$.

As the final part of this section we can obtain the relation of mapping $\partial_i$ and generator $x_j$ by using the Leibniz property of ${\sf d}$ with the differential calculus given by \eqref{31}
\begin{equation}\label{40}
\partial_ix_j=\delta_{ij}+\eta(\varepsilon_j, \varepsilon_i)x_j\partial_{i}, 1\leq i,j \leq n,
\end{equation}
which complete the scheme of Weyl algebra corresponding to ${\mathcal{A}}_q(n)$ together with the relation \eqref{22}.

\section{ Space of Maurer-Cartan 1-Forms on ${\mathcal{A}}_q(n)$} 
In the framework of the Hopf algebra ${\mathcal{A}}_q(n) $, the right-invariant Maurer-Cartan
form for any  $f\in {\mathcal{A}}_q(n)$ is defined by through the formula \cite{journal-3:1989}
\begin{equation}\label{41}
 w_f:=m(( {\sf d}\otimes S) \Delta(f))
 \end{equation}
where $m$ stands for the multiplication. Thus, for the noncommuting
coordinates of $ {\mathcal{A}}_q(n)$,
 \begin{equation*}
\begin{aligned}
 \omega_{x_1}& = m(( {\sf d}\otimes S) \Delta(x_1))= m(( {\sf d}\otimes S)(x_1\otimes
x_1)) = m({\sf d}(x_1)\otimes S(x_1))=  {\sf d}x_1x_1^{-1}, \nonumber \\
\omega_{x_i} &= m(( {\sf d}\otimes S) \Delta(x_i))= m(( {\sf d}\otimes S)(x_1\otimes
x_i+x_i\otimes x_1))={\sf d}x_i ~x_1^{-1}- {\sf d}x_1 ~x_1^{-1}x_ix_1^{-1},
\end{aligned}
\end{equation*}
where $1<i \leq n$.
Let $\omega_i=\omega_{x_i}$. Then \eqref{41} implies that for any $f\in {\mathcal{A}}_q(n)$, $\omega_f$ could be written
as a linear combination of  all $\omega_i$'s of the form $\omega_f=f_1\omega_1+\cdot \cdot \cdot +f_n \omega_n$ where $f_i \in {\mathcal{A}}_q(n), i=1,2,3$. Now, using \eqref{14},  \eqref{31} and \eqref{39}, we obtain some relations, which will be used in the following sections, such
as commutation relations of $\omega_i$'s with the generators $x_i$'s as
\begin{equation}\label{42}
\begin{aligned}
&x_i \omega_{1} = \omega_{x_1} x_i, ~ 1\leq i \leq n \\
&  x_i \omega_{j}= q^{j-1} \omega_ {j}x_i,~1\leq i\leq n,~1<j \leq n,
	\end{aligned} 
	\end{equation}
and one between any $\omega_{i}$ and $\omega_{j}$ as follows:

\begin{equation}\label{43}
	\omega_{i} \wedge \omega_j=-(1-\delta_{ij})\omega_{j} \wedge \omega_{i}
 \end{equation}
\section{Vector Fields Corresponding to the Maurer-Cartan Forms} 
In this section we will construct the quantum Lie algebra
of vector fields, denoted by $\mathcal{T}$, corresponding to the right-invariant Maurer-Cartan
forms obtained in the Section 4. In order to obtain the generators of $\mathcal{T}$ in
terms of the derivation mappings $\partial_i$  we first write the
Maurer-Cartan forms as follows
\begin{equation}\label{44}
{\sf d}x_1 = \omega_{1}x_1, ~{\sf d}x_i = \omega_{1}x_i + \omega_{i}x_1, ~ 1<i \leq n
\end{equation}
 Let us consider the
differential ${\sf d}$ in terms of the Maurer-Cartan forms and
 the generators of $ \mathcal{T}$ :
\begin{equation}\label{45}
{\sf d} =\sum_{i=1}^n \omega_{i}T_{i} 
\end{equation}
 where $T_{i}$'s are  generators of  $\mathcal{T}$. By inserting \eqref{44} to
the expression
\begin{equation}\label{46}
{\sf d}=\sum_{i=1}^n{\sf d}x_i\partial_i\nonumber
\end{equation}
we obtain the generators as the following vector fields:
\begin{equation}\label{47}
\begin{aligned}
&T_{1} \equiv \sum_{i=1}^n x_i \partial_{i} \\
&T_{i}\equiv x_1\partial_{i}, ~ 1<i \leq n
\end{aligned}
\end{equation}
 Together with the relations \eqref{14}, \eqref{22} and \eqref{40}, \eqref{47} implies that commutation relation between two generators $T_i$ and $T_j$ is of the form
\begin{equation}\label{48}
T_{i} T_{j} - T_{j} T_{i}  = 0,~ 1\leq i,j \leq n
\end{equation}
The commutation relation \eqref{48} must be
compatible with monomials of the algebra $\mathcal{A}$. To realize this, we
get the following commutation relations between the generators of $\mathcal{T}$
and the coordinates $x_i$'s due to the relations in \eqref{14} and \eqref{40} :
\begin{equation*}
\begin{aligned}
&T_{1}~x_j = x_j + x_j~T_{1}, ~ 1\leq j \leq n \\
&T_{i}~x_j = \delta_{ij}x_1+q^{i-1}~x_j~T_{i}, ~1< i\leq n,~1\leq j \leq n 
\end{aligned}
\end{equation*}
\begin{lem}
Let $f$ be any monomial of the form $x^{\alpha}$ in ${\mathcal{A }}_q(n)$. For any $g \in {\mathcal{A}}_q(n)$,  the vector fields $T_{i}$'s have the following $q$-deformed Leibniz rules when they act on $fg$ :
\begin{equation}\label{50}
T_{i}\left(fg\right)=T_{i}\left(f\right)~g+q^{\lambda_i}fT_{i}\left(g\right),~\lambda_i=(i-1)\sum_{k=1}^n\alpha_k
\end{equation}
\end{lem}
\textbf{Proof:}
 The relations in \eqref{42} result in the following relation between the form $\omega_i$ and the monomial $f$
\begin{equation}\label{51}
f\omega_{i}=q^{\lambda_i} \omega_{i}f. 
\end{equation}
From the Leibniz rule and the exterior differential operator ${\sf d}$ given by \eqref{45}, we have 
\begin{equation}\label{52}
\left(\sum_{i=1}^n{\omega_{i}T_{i}} \right)\left(fg\right)=\left(\sum_{i=1}^n{\omega_{i}T_{i}} \right) \left( f \right)g +f\left(\sum_{i=1}^n{\omega_{i}T_{i}} \right)(g)
\end{equation}
Inserting \eqref{51} to \eqref{52} and collecting according to $\omega_{i}$, we obtain
\begin{equation}\label{53}
\left(\sum_{i=1}^n{\omega_{i}T_{i}} \right)\left(fg\right)=\sum_{i=1}^n{\omega_{i}\left(T_{i} ( f )g +q^{\lambda_i}fT_{i} (g)\right)}
\end{equation}
which results in \eqref{50}.

\begin{thm}
We have the following $q$-deformed coproduct for the vector fields $T_{i}$, which is consistent with the $q$-deformed Leibniz rule \eqref{50}:
\begin{equation}\label{54}
 \Delta(T_{i})  =  T_{i} \otimes {\bf 1} +  q^{(i-1)T_{1}} \otimes T_{i}
 \end{equation}
\end{thm}
\textbf{Proof:}
If we consider tensor product of  the form
\begin{equation}\label{55}
(X\otimes Y)(f\otimes g)=X(f)\otimes Y(g) ,~ f,g \in {\mathcal{A}},~ X,Y \in \mathcal{T}
\end{equation}
and $m(\Delta(X)(f\otimes g)):=X(fg)$, then we have, from the $q$-deformed Leibniz rule \eqref{50}, the q-deformed coproduct \eqref{55}. Notice that by the action \eqref{21},
 the vector field $T_1$ acts on the monomial $f=x^{\alpha}$
as follows
\begin{equation}\label{56}
{T_1}(f)=\left (\sum_{i=1}^n{x_i \partial_{x_i}} \right)(f)=\left( \sum_{i=1}^n{\alpha_i}\right).f\nonumber
\end{equation}.
Since the action rule \eqref{21} holds also for a negative power of $x_1$, the action of $T_i$ can be extended to a negative power of $x_1$. 

Finally, in order to introduce Hopf algebra for the universal enveloping algebra $\mathcal{U}(\mathcal{T})$, we obtain the counit and antipode corresponding to the coproduct given in \eqref{54} as 
\begin{equation}\label{57}
\begin{aligned}
&\varepsilon(T_i)  =  0, ~1\leq i \leq n \\
& S(T_i) =  - {q^{-(i-1)T_1}} T_i,~1\leq i \leq n
\end{aligned}
\end{equation}
 Notice that the Hopf algebra derived by \eqref{54} is $q$-deformed version of the usual Hopf algebra with the primitive coproduct $\Delta(T_i)=T_i\otimes {\bf 1}+{\bf 1}\otimes T_i$ obtained in the classical case $q=1$.
\section{Conclusion} 
In this study, a quantum $n$-space whose commutation relations are in the sense of Manin has been presented in such a manner that the relevant algebra can be equipped by a cocommutative Hopf algebra structure. Morever some linear mappings of the quantum $n$-space are described in such a way that these mappings reduce to the usual derivation operators in the case of the classical limit $q\rightarrow 1$, and this situation can be expressed by $\sigma$-twisted derivations where $\sigma$ is an algebra automorphism. Further, we show that a bicovariant differential calculus on the quantum $n$-space can be constructed via these twisted derivations. Finally, the right invariant Maurer-Cartan forms and the corresponding vector fields are given, and it is seen that the algebra of these vector fields has a non-cocommutative Hopf algebra structure. 
\section{Acknowledgments}
I would like to express my deep gratitude to Turkish Scientific and Technical
Research Council for supporting in part this work and the
referees for their valuable suggestions concerning the material in this paper.

\end{document}